\documentclass[12pt]{article}

\title{Cylindrical contact homology for dynamically convex contact forms in three dimensions}
\author{Michael Hutchings\footnote{Partially supported by NSF grant DMS-1105820.}\;  and Jo Nelson\footnote{Supported by NSF grant DMS-1303903, the Bell Companies Fellowship, the Charles Simonyi Endowment, the Fund for Mathematics at the Institute for Advanced Study, and the B\&B chez Katrin.}\;}
\date{}

\usepackage{amssymb}
\usepackage{latexsym}
\usepackage{amsmath}
\usepackage{amsthm}
\usepackage{amscd}
\usepackage[dvips]{graphics}
\usepackage{overpic}

\newcommand{\mc}[1]{{\mathcal #1}}

\numberwithin{equation}{section}

\newtheorem{theorem}{Theorem}[section]
\newtheorem{proposition}[theorem]{Proposition}

\newtheorem{lemma}[theorem]{Lemma}
\newtheorem{lemma-definition}[theorem]{Lemma-Definition}

\newtheorem{conjecture}[theorem]{Conjecture}

\theoremstyle{definition}
\newtheorem{definition}[theorem]{Definition}
\newtheorem{remark}[theorem]{Remark}

\newtheorem{example}[theorem]{Example}

\newcommand{\floor}[1]{\left\lfloor #1 \right\rfloor}
\newcommand{\ceil}[1]{\left\lceil #1 \right\rceil}

\newcommand{\C}{{\mathbb C}}
\newcommand{\Q}{{\mathbb Q}}
\newcommand{\R}{{\mathbb R}}

\newcommand{\Z}{{\mathbb Z}}

\newcommand{\op}{\operatorname}

\newcommand{\M}{\mc{M}}

\newcommand{\Ker}{\op{Ker}}

\newcommand{\tensor}{\otimes}

\newcommand{\CZ}{\op{CZ}}

\newcommand{\bpm}{\begin{pmatrix}}
\newcommand{\epm}{\end{pmatrix}}

\begin{document}

\setcounter{tocdepth}{2}

\maketitle

\begin{abstract}
We show that for dynamically convex contact forms in three dimensions, the cylindrical contact homology differential $\partial$ can be defined by directly counting holomorphic cylinders for a generic almost complex structure, without any abstract perturbation of the Cauchy-Riemann equation. We also prove that $\partial^2=0$. Invariance of cylindrical contact homology in this case can be proved using $S^1$-dependent almost complex structures, similarly to work of Bourgeois-Oancea; this will be explained in another paper.
\end{abstract}

\tableofcontents

\section{Introduction and statement of results}

\subsection{Introduction}

Cylindrical contact homology, introduced by Eliashberg-Givental-Hofer \cite{egh}, is in principle an invariant of contact manifolds $(Y,\xi)$ that admit a contact form $\lambda$ without contractible Reeb orbits of certain gradings. The cylindrical contact homology of $(Y,\xi)$ is defined by choosing a nondegenerate such contact form $\lambda$ and taking the homology of a chain complex over $\Q$ which is generated by ``good'' Reeb orbits, and whose differential $\partial$ counts $J$-holomorphic cylinders in $\R\times Y$ for a suitable almost complex structure $J$. Unfortunately, in many cases there is no way to choose $J$ so as to obtain the transversality for holomorphic cylinders needed to define $\partial$ and to show that $\partial^2=0$ and that the homology is invariant. Thus, to define cylindrical contact homology in general, some kind of ``abstract perturbation'' of the $J$-holomorphic curve equation is needed, for example using polyfolds or Kuranishi structures, and this is still a work in progress. 

Although such abstract methods should be able to define contact homology in general, for computations and applications it is often desirable to have a more explicit geometric definition of the chain complex, in those special situations when this is possible. The goal of this paper is to show that for ``dynamically convex'' contact forms $\lambda$ in three dimensions, and for generic almost complex structures $J$, one can in fact define the differential $\partial$ by counting $J$-holomorphic cylinders without any abstract perturbation. We also show that $\partial^2=0$. (When $\pi_1(Y)$ contains torsion we make one additional assumption, which can be removed if a certain technical conjecture holds.)

Previously, the paper \cite{bce} claimed to show that cylindrical contact homology is well-defined and invariant for dynamically convex contact forms on $S^3$. However the argument had two gaps: First, a certain kind of breaking of index $2$ cylinders that could potentially interfere with the compactness argument in the proof that $\partial^2=0$ was not considered, see Proposition~\ref{prop:1neg}(c) below. Second, $S^1$-dependent almost complex structures were used to guarantee transversality of the moduli spaces of holomorphic cylinders. However, breaking the $S^1$ symmetry invalidates the gluing property needed to prove $\partial^2=0$ and the chain map and chain homotopy equations\footnote{One can correct for this failure of gluing, but one is then naturally led to a ``Morse-Bott'' version of the chain complex, with two generators for each Reeb orbit, analogous to \cite{bo09}. The homology of this Morse-Bott chain complex is not the desired cylindrical contact homology, but rather a ``non-equivariant'' version of it. The cylindrical contact homology that we want can be regarded as an ``$S^1$-equivariant'' version of the latter homology, and recovering this requires an additional construction as in \cite{bo12}.}.

We deal with the first issue by using intersection theory of holomorphic curves to show that the troublesome breaking cannot occur for generic $J$. To deal with second issue, we use index calculations to show that one can already obtain the transversality needed to define $\partial$ and prove that $\partial^2=0$ using a generic  almost complex structure, without breaking the $S^1$ symmetry.

Next one would like to show that cylindrical contact homology is an invariant of three-manifolds with contact structures that admit dynamically convex contact forms, by counting holomorphic cylinders in cobordisms to define chain maps and chain homotopies. It turns out that generic ($S^1$-independent) almost complex structures do not give sufficient transversality to count index zero cylinders in cobordisms to define chain maps, see Remark~\ref{rem:special} below. Instead, as we explain in the sequel \cite{sequel}, one can prove this topological invariance using $S^1$-dependent almost complex structures similarly to \cite{bo12}. In fact, the proof of invariance shows that cylindrical contact homology lifts to an invariant with integer coefficients, see Remark~\ref{rem:chz}.

\subsection{Holomorphic cylinders}

We now set up some notation for holomorphic cylinders in the symplectization of a contact three-manifold.

Let $Y$ be a closed three-manifold with a contact form $\lambda$. Let $\xi=\Ker(\lambda)$ denote the associated contact structure, and let $R$ denote the associated Reeb vector field.

A {\em Reeb orbit\/} is a map $\gamma:\R/T\Z\to Y$ for some $T>0$ such that $\gamma'(t)=R(\gamma(t))$, modulo reparametrization. We do not assume that $\gamma$ is an embedding. For a Reeb orbit as above, the linearized Reeb flow for time $T$ defines a symplectic linear map
\begin{equation}
\label{eqn:lrt}
P_\gamma:(\xi_{\gamma(0)},d\lambda) \longrightarrow (\xi_{\gamma(0)},d\lambda).
\end{equation}
The Reeb orbit $\gamma$ is {\em nondegenerate\/} if $P_\gamma$ does not have $1$ as an eigenvalue.  The contact form $\lambda$ is called nondegenerate if all Reeb orbits are nondegenerate; generic contact forms have this property. Fix a nondegenerate contact form below.

A (nondegenerate) Reeb orbit $\gamma$ is {\em elliptic\/} if $P_\gamma$ has eigenvalues on the unit circle, {\em positive hyperbolic\/} if $P_\gamma$ has positive real eigenvalues, and {\em negative hyperbolic\/} if $P_\gamma$ has negative real eigenvalues. If $\tau$ is a homotopy class of trivializations of $\xi|_\gamma$, then the {\em Conley-Zehnder index\/} $\CZ_\tau(\gamma)\in\Z$ is defined, see the review in \S\ref{sec:imc}. The parity of the Conley-Zehnder index does not depend on the choice of trivialization $\tau$, and is even when $\gamma$ is positive hyperbolic and odd otherwise.

We say that an almost complex structure $J$ on $\R\times Y$ is {\em $\lambda$-compatible\/} if $J(\xi)=\xi$; $d\lambda(v,Jv)>0$ for nonzero $v\in\xi$; $J$ is invariant under translation of the $\R$ factor; and $J(\partial_s)=R$, where $s$ denotes the $\R$ coordinate. Fix such a $J$.

If $\gamma_+$ and $\gamma_-$ are Reeb orbits, we consider $J$-holomorphic cylinders between them, namely maps $u:\R\times S^1\to\R\times Y$ such that
\[
\partial_su+J\partial_tu=0,
\]
$\lim_{s\to\pm_\infty}\pi_\R(u(s,t))=\pm\infty$, and $\lim_{s\to\pm\infty}\pi_Y(u(s,\cdot))$ is a parametrization of $\gamma_\pm$. Here $\pi_\R$ and $\pi_Y$ denote the projections from $\R\times Y$ to $\R$ and $Y$ respectively. We say that $u$ has a ``positive end at $\gamma_+$'' and a ``negative end at $\gamma_-$''. We declare two such maps to be equivalent if they differ by translation and rotation of the domain $\R\times S^1$, and we denote the set of equivalence classes by $\M^J(\gamma_+,\gamma_-)$. Note that $\R$ acts on $\M^J(\gamma_+,\gamma_-)$ by translation of the $\R$ factor in $\R\times Y$.

Given $u$ as above, we define its {\em Fredholm index\/} by
\[
\op{ind}(u) = \CZ_\tau(\gamma_+) - \CZ_\tau(\gamma_-) + 2c_\tau(u).
\]
Here $\tau$ is a trivialization of $\xi$ over $\gamma_+$ and $\gamma_-$, and $c_\tau(u)$ denotes the relative first Chern class $c_1(u^*\xi,\tau)$, see \cite[\S2.5]{ir} or \cite[\S3.2]{bn}; the relative first Chern class vanishes when the trivialization $\tau$ extends to a trivialization of $u^*\xi$. The significance of the Fredholm index is that if $J$ is generic and $u$ is somewhere injective, then $\M^J(\gamma_+,\gamma_-)$ is naturally a manifold near $u$ of dimension $\op{ind}(u)$. Let $\mc{M}^J_k(\gamma_+,\gamma_-)$ denote the set of $u\in\mc{M}^J(\gamma_+,\gamma_-)$ with $\op{ind}(u)=k$.

\subsection{Cylindrical contact homology}

As above, let $\lambda$ be a nondegenerate contact form on the closed three-manifold $Y$, and let $J$ be a generic $\lambda$-compatible almost complex structure on $\R\times Y$. In the absence of certain kinds of contractible Reeb orbits, one would like to define cylindrical contact homology as follows. (The original definition is in \cite{egh}; we are using some different notation and conventions.)

A Reeb orbit $\gamma$ is said to be {\em bad\/}\footnote{In general, in any number of dimensions, a nondegenerate Reeb orbit is bad when it is an even multiple cover of another Reeb orbit whose Conley-Zehnder index has opposite parity.} if it is an even degree multiple cover of a negative hyperbolic orbit; otherwise $\gamma$ is called {\em good\/}.

Define $CC^\Q(Y,\lambda,J)$ to be the vector space over $\Q$ generated by the good Reeb orbits. One would like to define an operator
\[
\delta:CC^\Q(Y,\lambda,J) \longrightarrow CC^\Q(Y,\lambda,J)
\]
by the equation
\begin{equation}
\label{eqn:delta}
\delta\alpha = \sum_\beta\sum_{u\in\mc{M}^J_1(\alpha,\beta)/\R}\frac{\epsilon(u)}{d(u)}\beta.
\end{equation}
Here $\epsilon(u)\in\{\pm1\}$ is a sign associated to $u$ via a system of coherent orientations as in \cite{bm,fh}, while $d(u)\in\Z^{>0}$ is the covering multiplicity of $u$ (which is $1$ if and only if $u$ is somewhere injective). The definition \eqref{eqn:delta} only makes sense if the moduli spaces $\mc{M}^J_1(\alpha,\beta)/\R$ are compact and cut out transversely. 

Define another operator
\[
\kappa: CC^\Q(Y,\lambda,J) \longrightarrow CC^\Q(Y,\lambda,J)
\]
by
\[
\kappa(\alpha) = d(\alpha)\alpha,
\]
where $d(\alpha)\in\Z^{>0}$ denotes the covering multiplicity of $\alpha$. By counting ends of the moduli spaces $\mc{M}_2^J(\alpha,\beta)/\R$, one expects (in the absence of certain contractible Reeb orbits) to obtain the equation
\[
\delta\kappa\delta=0.
\]
This equation implies that
\[
\partial = \delta\kappa
\]
is a differential on $CC^\Q(Y,\lambda,J)$. 
The homology of the resulting chain complex $(CC^\Q(Y,\lambda,J),\partial)$ is the cylindrical contact homology $CH^\Q(Y,\lambda,J)$.

Note that a different choice of coherent orientations will lead to different signs in the differential, but the chain complexes will be canonically isomorphic. Note also that some papers use a different convention in which the differential, in our notation, is $\kappa\delta$ instead of $\delta\kappa$. The operator $\kappa$ defines an isomorphism between these two chain complexes, because $(\kappa\delta)\kappa=\kappa(\delta\kappa)$, cf.\ \cite[Rem.\ 3.4]{bee}.

\subsection{The main result}

Let $Y$ be a closed three-manifold with a nondegenerate contact form $\lambda$. Suppose that $\pi_2(Y)=0$, or more generally that $c_1(\xi)|_{\pi_2(Y)}=0$. Then for each contractible Reeb orbit $\gamma$, we can define the Conley-Zehnder index of $\gamma$ by $\CZ(\gamma)=\CZ_\tau(\gamma)$, where $\tau$ is a trivialization of $\xi|_\gamma$ which extends to a trivialization of $\xi$ over a disk bounded by $\gamma$.

\begin{definition}
(cf.\ \cite{hwz99})
Let $\lambda$ be a nondegenerate contact form on a closed three-manifold $Y$. We say that $\lambda$ is {\em dynamically convex\/} if either:
\begin{itemize}
\item
$\lambda$ has no contractible Reeb orbits, or
\item
$c_1(\xi)|_{\pi_2(Y)}=0$, and every contractible Reeb orbit $\gamma$ has $\CZ(\gamma)\ge 3$.
\end{itemize}
\end{definition}

\begin{example}
\cite{hwz99}
If $Y$ is a compact star-shaped (i.e.\ transverse to the radial vector field) hypersurface in $\R^4$, then
\[
\lambda = \frac{1}{2}\sum_{k=1}^2(x_kdy_k-y_kdx_k)
\]
restricts to a contact form on $Y$. If $Y$ is convex, then $\lambda$ is dynamically convex (if it is nondegenerate, which holds for generic $Y$).
\end{example}

The main result of this paper is the following:

\begin{theorem}
\label{thm:main}
Let $\lambda$ be a nondegenerate, dynamically convex contact form on a closed three-manifold $Y$. Suppose further that:
\begin{description}
\item{(*)}
A contractible Reeb orbit $\gamma$ has $\CZ(\gamma)=3$ only if $\gamma$ is embedded.
\end{description}
Then for generic $\lambda$-compatible almost complex structures $J$ on $\R\times Y$, the operator $\delta$ in \eqref{eqn:delta} is well defined and satisfies $\delta\kappa\delta=0$, so $(CC^\Q(Y,\lambda,J),\partial)$ is a well-defined chain complex where $\partial=\delta\kappa$.
\end{theorem}

\begin{remark}[\bf on the hypotheses]
\begin{description}
\item{(a)}
The hypothesis (*) automatically holds when $\pi_1(Y)$ contains no torsion, because if $\gamma$ is a contractible Reeb orbit and $\gamma^d$ denotes its $d$-fold cover, then $\CZ(\gamma^d)\ge d\CZ(\gamma)-d+1$.
\item{(b)}
In general, the hypothesis (*) can be removed from Theorem~\ref{thm:main} assuming a certain technical conjecture on the asymptotics of holomorphic curves, see Remark~\ref{remark:technical}.
\item{(c)}
We expect that with similar technical work, the hypothesis of dynamical convexity can be weakened to the hypothesis that all contractible Reeb orbits $\gamma$ have $\CZ(\gamma)\ge 2$, provided that whenever $\CZ(\gamma)=2$, the count of holomorphic planes asymptotic to $\gamma$ is zero (this last condition will hold for example if $(Y,\lambda)$ has an exact filling).
\end{description}
\end{remark}

\begin{remark}[\bf grading]
The chain complex $CC^\Q$ splits into a direct sum of subcomplexes according to the homotopy classes of the Reeb orbits in the free loop space of $Y$. The condition $c_1(\xi)|_{\pi_2(Y)}=0$ (which we are assuming when there are contractible Reeb orbits) implies that each of these subcomplexes has a relative $\Z$-grading. The subcomplex generated by contractible Reeb orbits has a canonical absolute $\Z$-grading by $\CZ-1$.
\end{remark}
 
The following theorem will be proved in the sequel \cite{sequel}:

\begin{theorem}
\label{thm:invariance}
The cylindrical contact homology $CH^\Q$ is an invariant of pairs $(Y,\xi)$ where $Y$ is a closed three-manifold, and $\xi$ is a contact structure on $Y$ that admits a dynamically convex contact form $\lambda$ satisfying the condition (*).
\end{theorem}

\begin{remark}[\bf local and sutured versions]
Although Theorems~\ref{thm:main} and \ref{thm:invariance} are stated for closed three-manifolds, their proofs also work for open three-manifolds in situations where Gromov compactness still holds. For example:

(a) Assuming suitable transversality, Hryniewicz-Macarini \cite{hm} define the {\em local contact homology\/} $CH^\Q(\lambda,\gamma^m)$, where $\gamma$ is an embedded (possibly degenerate) Reeb orbit in a contact manifold $(Y,\lambda)$ such that all iterates of $\gamma$ are isolated Reeb orbits, and $m$ is a positive integer. To define this, let $N$ be a tubular neighborhood of $\gamma$ in $Y$, and let $\lambda'$ be a nondegenerate contact form in $N$ obtained by a perturbation of $\lambda$ which is small with respect to $m$ (in particular $\lambda'$ has no short contractible Reeb orbits). Then
$CH^\Q(\lambda,\gamma^m)$ is the cylindrical contact homology of $\lambda'$ in $N$ for Reeb orbits that wind $m$ times around $N$. When $\dim(Y)=3$, the proof of Theorem~\ref{thm:main} shows that the local contact homology chain complex is defined for generic $J$, and the proof of Theorem~\ref{thm:invariance} shows that the homology depends only on the contact form in a neighborhood of $\gamma$.

(b) The proofs of Theorems~\ref{thm:main} and \ref{thm:invariance} give a construction of cylindrical contact homology on sutured contact three-manifolds \cite{cghh} admitting dynamically convex contact forms satisfying (*).
\end{remark}

\begin{remark}[\bf coefficients]
The differential $\partial=\delta\kappa$, as well as the alternate differential $\kappa\delta$, in fact have integer coefficients (because the covering multiplicity of a holomorphic cylinder always divides the covering multiplicities of the Reeb orbits at its ends). However we do not expect the homologies of these differentials over $\Z$ to be invariant or isomorphic to each other in general.
\end{remark}

\begin{remark}[\bf relation with Bourgeois-Oancea]
\label{rem:chz}
(a)
 The proof of Theorem~\ref{thm:invariance} will show that there is in fact an invariant $CH^\Z$ (of pairs $(Y,\xi)$ as in the theorem) which is the homology of a chain complex $CC^\Z$ over $\Z$, such that $CH^\Q=CH^\Z\tensor\Q$. However the definition of $CC^\Z$ is quite different; it starts with a ``Morse-Bott'' chain complex with two generators for each (good or bad) Reeb orbit as in \cite{bo09}, and then passes to an $S^1$-equivariant version similarly to \cite{bo12}. We expect that $CH^\Z$ agrees with the invariant $SH_*^{S^1}(Y,\xi)$ defined in \cite[\S4.1.2]{bo12}; the latter is a version of $S^1$-equivariant symplectic homology which is an invariant of contact structures that admit dynamically convex contact forms\footnote{Here $SH_*^{S^1}(Y,\xi)$ denotes the sum over all free homotopy class of loops $c$ of the invariant $SH_*^{c,S^1}(Y,\xi)$ defined in \cite[\S4.1.2]{bo12}.}.

(b)
It is shown in \cite[Thm.\ 4.3]{bo12} that if $\lambda$ is dynamically convex and $J$ satisfies certain transversality assumptions, then $CH^\Q(Y,\lambda,J)$ is isomorphic to $SH_*^{S^1}(Y,\xi)\tensor \Q$. However we cannot invoke this theorem to prove the invariance in Theorem~\ref{thm:invariance}, because given a dynamically convex $\lambda$, it is often impossible to find any $J$ satisfying the transversality hypotheses\footnote{
For example, if the shortest Reeb orbit $\gamma$ is contractible and elliptic and has $\CZ(\gamma)=3$, and if the count of holomorphic planes asymptotic to $\gamma$ with a point constraint is nonzero, then for any $J$, holomorphic buildings as in Proposition~\ref{prop:1neg}(c)(iii) will exist, which violates the hypotheses of \cite[Thm.\ 4.3]{bo12}.} of \cite[Thm.\ 4.3]{bo12}.
\end{remark}

The rest of this paper is organized as follows. \S\ref{sec:index} carries out some index calculations which are needed in the compactness arguments in the proof of Theorem~\ref{thm:main}. \S\ref{sec:breaking} rules out a certain kind of breaking of index 2 cylinders which, if it happened, would cause a problem for the proof that $\partial^2=0$. Finally, \S\ref{sec:gluing} discusses transversality and gluing and completes the proof of Theorem~\ref{thm:main}.

\section{Index calculations}
\label{sec:index}

In this section we carry out some index calculations which are needed to rule out various bad degenerations in the compactness arguments to prove that $\partial$ is defined and $\partial^2=0$.
In this section, we always assume that $\lambda$ is a nondegenerate contact form on a three-manifold $Y$, and $J$ is a $\lambda$-compatible almost complex structure on $\R\times Y$. We do not assume that $\lambda$ is dynamically convex or that $J$ is generic unless otherwise stated.

\subsection{Estimates on the index of multiple covers}
\label{sec:imc}

We first obtain some estimates on the Fredholm index of certain kinds of multiply covered curves, without assuming dynamical convexity.

Let $u$ be a $J$-holomorphic curve in $\R\times Y$ with positive ends at Reeb orbits $\alpha_1,\ldots,\alpha_k$ and negative ends at Reeb orbits $\beta_1,\ldots,\beta_l$. Note that the Reeb orbits $\alpha_i$ and $\beta_j$ are not necessarily embedded. Recall that the Fredholm index of $u$ is given by the formula
\begin{equation}
\label{eqn:ind}
\op{ind}(u)=-\chi(u) + 2c_\tau(u) + \sum_{i=1}^k\CZ_\tau(\alpha_i) - \sum_{j=1}^l\CZ_\tau(\beta_j).
\end{equation}
Here $\chi(u)$ denotes the Euler characteristic of the domain of $u$, so if $u$ is irreducible of genus $g$ then
\begin{equation}
\label{eqn:chi}
\chi(u) = 2-2g-k-l.
\end{equation}
Also $\tau$ is a trivialization of $\xi$ over the Reeb orbits $\alpha_i$ and $\beta_j$; and $c_\tau$ and $\CZ_\tau$ denote the relative first Chern class of $u^*\xi$ and Conley-Zehnder index with respect to $\tau$ as before.

In three dimensions there is a useful explicit formula for the Conley-Zehnder index:
\[
\CZ_\tau(\gamma) = \floor{\theta} + \ceil{\theta}
\]
where $\theta$ denotes the ``rotation number'' of $\gamma$ with respect to $\tau$. If $\gamma$ is hyperbolic, then $\theta$ is the number of times that the eigenspaces of the linearized return map \eqref{eqn:lrt} rotate with respect to $\tau$ as one goes around $\gamma$; this is an integer if $\gamma$ is positive hyperbolic and an integer plus $1/2$ if $\gamma$ is negative hyperbolic. If $\gamma$ is elliptic then $\theta$ is an irrational number (due to our assumption that all possibly multiply covered Reeb orbits are nondegenerate), see \cite[\S3.2]{bn}. Changing the trivialization $\tau$ will shift the rotation number $\theta$ by an integer. Also, if $m$ is a positive integer and if $\gamma^m$ denotes the Reeb orbit that is a $m$-fold multiple cover of $\gamma$, then
\begin{equation}
\label{eqn:CZiteration}
\CZ_\tau(\gamma^m) = \floor{m\theta} + \ceil{m\theta},
\end{equation}
where $\theta$ still denotes the rotation number of $\gamma$ with respect to $\tau$.

We define a {\em trivial cylinder\/} to be a $J$-holomorphic cylinder $\R\times\gamma$ in $\R\times Y$ where $\gamma$ is a Reeb orbit. We do not require $\gamma$ to be embedded.

\begin{lemma}
\label{lem:ht}
\cite[Lem.\ 1.7]{ht1}
If $u$ is a $J$-holomorphic curve in $\R\times Y$ which is a branched cover\footnote{In this paper, ``branched covers'' also include coverings without branch points.} of a trivial cylinder, then $\op{ind}(u)\ge 0$.
\qed
\end{lemma}

\begin{lemma}
\label{lem:bestimate}
Let $u$ be a $J$-holomorphic curve in $\R\times Y$ with genus zero, one positive end, and an arbitrary number of negative ends. Let $\overline{u}$ denote the somewhere injective curve covered by $u$, and let $d$ denote the covering multiplicity of $u$ over $\overline{u}$. Let $b$ denote the number of ramification points of this cover, counted with multiplicity. Then
\begin{equation}
\label{eqn:bestimate}
\op{ind}(u)\ge d\op{ind}(\overline{u}) + 2(1-d+b).
\end{equation}
\end{lemma}

\begin{proof}
We can choose the trivialization $\tau$ so that $c_\tau(\overline{u})=0$, which implies that $c_\tau(u)=0$ also. We fix such a trivialization $\tau$ and write $\CZ$ as a shorthand for $\CZ_\tau$.

Suppose that $\overline{u}$ has positive end at $\alpha$ and negative ends at $\beta_1,\ldots,\beta_k$. Since $\overline{u}$ must have genus zero, it follows from \eqref{eqn:ind} and \eqref{eqn:chi} that
\begin{equation}
\label{eqn:indubar}
\op{ind}(\overline{u}) = k-1 + \CZ(\alpha) - \sum_{i=1}^k\CZ(\beta_i).
\end{equation}

To obtain a similar formula for $\op{ind}(u)$, let $n$ denote the number of negative ends of $u$. By Riemann-Hurwitz we have
\[
\chi(u) = d\chi(\overline{u}) - b,
\]
which means that
\begin{equation}
\label{eqn:RH}
1-n = d(1-k) - b.
\end{equation}
Let $\gamma_1,\ldots,\gamma_n$ denote the Reeb orbits at which $u$ has negative ends (these are covers of $\beta_1,\ldots,\beta_k$). By the iteration formula \eqref{eqn:CZiteration} for the Conley-Zehnder index, we have
\[
\left|\CZ\left(\gamma^{m_1+m_2}\right) - \CZ\left(\gamma^{m_1}\right) - \CZ\left(\gamma^{m_2}\right) \right|\le 1,
\]
assuming that all Conley-Zehnder indices are computed using the same trivialization of $\xi|_\gamma$. It follows from this that
\begin{equation}
\label{eqn:CZalphad}
\CZ(\alpha^d) \ge d \CZ(\alpha) - (d-1)
\end{equation}
and
\begin{equation}
\label{eqn:CZgammai}
\sum_{i=1}^n\CZ(\gamma_i) \le d\sum_{i=1}^k\CZ(\beta_i) + (dk-n).
\end{equation}
Using the above two inequalities, then using equation \eqref{eqn:indubar}, and finally using equation \eqref{eqn:RH}, we obtain
\[
\begin{split}
\op{ind}(u) &= n-1 + \CZ(\alpha^d) - \sum_{i=1}^n \CZ(\gamma_i)\\
&\ge n-1 + d \CZ(\alpha)-(d-1) -d\sum_{i=1}^k\CZ(\beta_i) - (dk-n)\\
&= d\op{ind}(\overline{u}) + 2(n-dk)\\
&= d\op{ind}(\overline{u}) +2(1-d+b).
\end{split}
\]
\end{proof}

We now use Lemma~\ref{lem:bestimate} to deduce two index estimates in the case when $J$ is generic which will be needed below.

\begin{lemma}
\label{lem:cylestimate}
Assume that $J$ is generic. Let $u$ be a genus zero holomorphic curve with one positive end and $n$ negative ends. Suppose that the somewhere injective curve $\overline{u}$ underlying $u$ is a nontrivial cylinder. Then
\[
\op{ind}(u) \ge n.
\]
\end{lemma}

\begin{proof}
Since $J$ is generic and $\overline{u}$ is nontrivial, it follows that $\op{ind}(\overline{u})>0$. Also observe that $b=n-1$.

If $\op{ind}(\overline{u})\ge 2$, then Lemma~\ref{lem:bestimate} implies that $\op{ind}(u)\ge 2n$ and we are done.

It remains to treat the case where $\op{ind}(\overline{u})=1$. In this case Lemma~\ref{lem:bestimate} gives
\begin{equation}
\label{eqn:improveme}
\op{ind}(u) \ge 2n-d,
\end{equation}
which in general might not be sufficient. However we can improve the inequality \eqref{eqn:improveme} as follows.

Let $\alpha$ and $\beta$ denote the Reeb orbits at which $\overline{u}$ has positive and negative ends respectively.
Since $\overline{u}$ has odd index, $\alpha$ or $\beta$ is positive hyperbolic.

If $\beta$ is positive hyperbolic, then by \eqref{eqn:CZiteration}, the inequality \eqref{eqn:CZgammai} can be replaced by the equality
\[
\sum_{i=1}^n\CZ(\gamma_i) = d\CZ(\beta).
\]
This allows us to add $d-n$ to the right side of \eqref{eqn:improveme} and we are done.

On the other hand, if $\alpha$ is positive hyperbolic, then \eqref{eqn:CZalphad} can be replaced by the equality
\[
\CZ(\alpha^d) = d\CZ(\alpha).
\]
This allows us to add $d-1$ to the right hand side of \eqref{eqn:improveme}, and since $d-1\ge d-n$ we are also done.
\end{proof}

\begin{lemma}
\label{lem:5estimate}
Assume that $J$ is generic. Let $u$ be a genus zero holomorphic curve with one positive end and $n>1$ negative ends. Suppose that $u$ is not a branched cover of a trivial cylinder. Then
\[
\op{ind}(u) \ge 5-2n.
\]
\end{lemma}

\begin{proof}
Let $\overline{u}$ denote the somewhere injective curve underlying $u$, let $d$ denote the covering multiplicity of $u$ over $\overline{u}$, and let $b$ denote the number of branch points counted with multiplicity. 

If $\overline{u}$ is a cylinder, then we are done by Lemma~\ref{lem:cylestimate} and the assumption that $n>1$.

It remains to treat the case where $\overline{u}$ has more than one negative end. Since $J$ is generic, we have $\op{ind}(\overline{u})\ge 1$. By Lemma~\ref{lem:bestimate} we then have
\[
\op{ind}(u) \ge 2-d+2b.
\]
By Riemann-Hurwitz as in equation \eqref{eqn:RH} we have
\[
n= d(k-1)+1+b
\]
where $k$ is the number of negative ends of $\overline{u}$.
It follows from the above inequality and equation that
\begin{equation}
\label{eqn:fati}
\op{ind}(u)+2n\ge d(2k-3)+4(b+1).
\end{equation}
Since $k>1$,
the right hand side of \eqref{eqn:fati} is at least $5$.
\end{proof}

We will also need the following facts about the index of multiply covered cylinders:

\begin{lemma}
\label{lem:cylindercoverindex}
Assume that $J$ is generic, and let $u$ be a nontrivial $J$-holomorphic cylinder in $\R\times Y$.  Let $\overline{u}$ denote the somewhere injective holomorphic cylinder underlying $u$. Then:
\begin{description}
\item{(a)}
$1\le \op{ind}(\overline{u}) \le \op{ind}(u)$.
\item{(b)}
If $\op{ind}(u) = 1$, and if $u$ has an end at a bad Reeb orbit, then the corresponding end of $\overline{u}$ is also at a bad Reeb orbit.
\end{description}
\end{lemma}

\begin{proof}
Let $\alpha$ and $\beta$ denote the Reeb orbits at which $\overline{u}$ has positive and negative ends. Let $d$ denote the covering multiplicity of $u$ over $\overline{u}$. Choose a trivialization $\tau$ of $\xi$ over $\alpha$ and $\beta$ so that $c_\tau(\overline{u})=0$. Then $c_\tau(u)=0$ also. Thus
\[
\op{ind}(u) = \CZ_\tau(\alpha^d) - \CZ_\tau(\beta^d)
\]
and 
\[
\op{ind}(\overline{u}) = \CZ_\tau(\alpha) - \CZ_\tau(\beta).
\]
We now obtain the conclusions of the lemma as follows:

(a) Since $J$ is generic and $\overline{u}$ is not a trivial cylinder, we have $\op{ind}(\overline{u})\ge 1$. Consequently we can choose the trivialization $\tau$ so that $\CZ_\tau(\alpha)\ge 0$ and $\CZ_\tau(\beta)\le 0$. It then follows from \eqref{eqn:CZiteration} that $\CZ_\tau(\alpha^d)\ge \CZ_\tau(\alpha)$ and $\CZ_\tau(\beta^d)\le \CZ_\tau(\beta)$, so $\op{ind}(u)\ge \op{ind}(\overline{u})$.

(b) Without loss of generality, the cover of $\alpha$ at which $u$ has an end is a bad Reeb orbit. We need to show that $\alpha$ itself is a bad Reeb orbit. If not, then $\alpha$ is negative hyperbolic and $d$ is even. Since $\op{ind}(\overline{u})$ is odd by part (a), the Reeb orbit $\beta$ must be positive hyperbolic.  Then both ends of $u$ are at positive hyperbolic orbits, so $\op{ind}(u)$ is even, which contradicts our hypothesis.
\end{proof}

\begin{remark}
\label{rem:special}
Lemma~\ref{lem:cylindercoverindex}(a) does not generalize to higher dimensional contact manifolds; when $\dim(Y)>3$, multiply covered cylinders may have index less than $1$, even when $J$ is generic. Also, a four-dimensional symplectic cobordism between three-dimensional contact manifolds may contain multiply covered $J$-holomorphic cylinders $u$ with $\op{ind}(u) < 0$ which cannot be eliminated by choosing $J$ generically. See \cite[Ex.\ 1.22]{jo1} for an example of this involving ellipsoids.
\end{remark}

\subsection{Low index buildings in the dynamically convex case}

We now classify certain holomorphic buildings of low index in the case when $\lambda$ is dynamically convex and $J$ is generic.

For our purposes, a ``holomorphic building'' is an $m$-tuple $(u_1,\ldots,u_m)$, for some positive integer $m$, of (possibly disconnected) $J$-holomorphic curves $u_i$ in $\R\times Y$, called ``levels''. Although our notation does not indicate this, the building also includes, for each $i\in\{1,\ldots,m-1\}$, a bijection between the negative ends of $u_i$ and the positive ends of $u_{i+1}$, such that paired ends are at the same Reeb orbit\footnote{One might also want a holomorphic building to include appropriate gluing data when Reeb orbits are multiply covered, but we will not need this.}. If $m>1$ then we assume that for each $i$, at least one component of $u_i$ is not a trivial cylinder. A ``positive end'' of the building $(u_1,\ldots,u_m)$ is a positive end of $u_1$, and a ``negative end'' of $(u_1,\ldots,u_m)$ is a negative end of $u_m$. The ``genus'' of the building $(u_1,\ldots,u_m)$ is the genus of the Riemann surface obtained by gluing together negative ends of the domain of $u_i$ and positive ends of the domain of $u_{i+1}$ by the given bijections (when this glued Riemann surface is connected). We define the Fredholm index of a holomorphic building by $\op{ind}(u_1,\ldots,u_m) = \sum_{i=1}^m\op{ind}(u_i)$.

\begin{proposition}
\label{prop:0neg}
Assume that $J$ is generic and $\lambda$ is dynamically convex. Suppose that $u=(u_1,\ldots,u_m)$ is a genus zero $J$-holomorphic building with one positive end and no negative ends. Then:
\begin{itemize}
\item $\op{ind}(u)\ge 2$.
\item If $\op{ind}(u)=2$, then $u$ has only one level (which of course is a plane).
\end{itemize}
\end{proposition}

\begin{proof}
We use induction on $m$.

Suppose $m=1$. If $u_1$ has its positive end at $\gamma$, then by equation \eqref{eqn:ind} we have $\op{ind}(u_1)=\CZ(\gamma)-1$. Thus the lemma follows from the definition of dynamically convex.

Now let $m>1$ and suppose the proposition is true for $m-1$. We need to show that $\op{ind}(u)>2$.

Let $n$ denote the number of negative ends of $u_1$.
The holomorphic building $(u_2,\ldots,u_m)$ is the union of $n$ genus zero holomorphic buildings, each having one positive end corresponding to one of the negative ends of $u_1$, and no negative ends. So by the inductive hypothesis we have
\[
\op{ind}(u) \ge \op{ind}(u_1) + 2n.
\]
To complete the proof we need to show that
\begin{equation}
\label{eqn:nts}
\op{ind}(u_1)+2n\ge 3.
\end{equation}

If $\overline{u_1}$ is a trivial cylinder, then we must have $n>1$, so \eqref{eqn:nts} follows from Lemma~\ref{lem:ht}. 
If $\overline{u_1}$ is a nontrivial cylinder, then \eqref{eqn:nts} follows from Lemma~\ref{lem:cylestimate}. If $\overline{u_1}$ is not a cylinder, then we must have $n>1$, so \eqref{eqn:nts} follows from Lemma~\ref{lem:5estimate}.
\end{proof}

\begin{proposition}
\label{prop:1neg}
Assume that $J$ is generic and $\lambda$ is dynamically convex. Suppose that $u=(u_1,\ldots,u_m)$ is a nontrivial genus zero $J$-holomorphic building with one positive end and one negative end. Then:
\begin{description}
\item{(a)} $\op{ind}(u)\ge 1$.
\item{(b)} If $\op{ind}(u)=1$ then $u$ has one level (which of course is a cylinder).
\item{(c)}
If $\op{ind}(u)=2$, then one of the following holds:
\begin{description}
\item{(i)}
$u$ has one level.
\item{(ii)}
$u$ has two levels which are both cylinders.
\item{(iii)}
$u=(u_1,u_2)$ where:
\begin{itemize}
\item
$u_1$ is an index zero degree $d_1+d_2$ branched cover of an embedded trivial cylinder $\R\times\gamma$ with two negative ends, one at $\gamma^{d_1}$ and one at $\gamma^{d_2}$.
\item
$u_2$ has two components; one component is the trivial cylinder $\R\times\gamma^{d_1}$, and the other component is an index two holomorphic plane with positive end at $\gamma^{d_2}$.
\end{itemize}
\end{description}
\end{description}
\end{proposition}

\begin{proof}
As in the proof of Proposition~\ref{prop:0neg}, we use induction on $m$.

If $m=1$ then the proposition follows from Lemma~\ref{lem:cylindercoverindex}(a). So suppose that $m>1$ and assume that the proposition is true for $m-1$. We need to show that $\op{ind}(u)\ge 2$, with equality only if (ii) or (iii) holds.

Let $n>0$ denote the number of negative ends of $u_1$. Then the holomorphic building $(u_2,\ldots,u_m)$ is the union of $n$ genus zero holomorphic buildings $B_1,\ldots,B_n$, each having one positive end. One of these buildings, say $B_1$, has one negative end, while $B_2,\ldots,B_n$ have no negative ends.

If $n=1$, then $B_1$ is nontrivial, so by the inductive hypothesis $\op{ind}(B_1)\ge 1$, with equality only if $B_1$ has one level which is a cylinder. On the other hand Lemma~\ref{lem:cylestimate} implies that $\op{ind}(u_1)\ge 1$. Thus $\op{ind}(u)\ge 2$, with equality only if (ii) holds.

Suppose now that $n>1$. If $B_1$ is trivial then $\op{ind}(B_1)=0$. Otherwise $\op{ind}(B_1)\ge 1$ by the inductive hypothesis. Either way, it follows from Proposition~\ref{prop:0neg} that
\begin{equation}
\label{eqn:B1triv}
\op{ind}(u) \ge \op{ind}(u_1)+2n-2,
\end{equation}
with equality only if $B_1$ is trivial and each of $B_2,\ldots,B_n$ is an index two plane.

If $\overline{u_1}$ is a trivial cylinder, then it follows from \eqref{eqn:B1triv} and Lemma~\ref{lem:ht} that $\op{ind}(u)\ge 2n-2\ge 2$. Equality holds only if $n=2$ and the equality conditions for \eqref{eqn:B1triv} are satisfied, which implies (iii).

If $\overline{u_1}$ is not a trivial cylinder, then it follows from \eqref{eqn:B1triv} and Lemma~\ref{lem:5estimate} that $\op{ind}(u)\ge 3$.
\end{proof}

\section{Ruling out bad breaking}
\label{sec:breaking}

The goal of this section is to prove the following proposition, which shows that for generic $J$, a sequence of holomorphic cylinders cannot converge to a building as in case (iii) of Proposition~\ref{prop:1neg}(c) with $d_2=1$.

\begin{proposition}
\label{prop:nobadbreak}
Let $Y$ be a closed three-manifold with a nondegenerate contact form $\lambda$, and let $J$ be a generic $\lambda$-compatible almost complex structure on $\R\times Y$. Let $u=(u_1,u_2)$ be a holomorphic building where:
\begin{itemize}
\item $u_1$ is an index zero pair of pants which is a degree $d+1$ branched cover of an embedded trivial cylinder $\R\times Y$ with positive end at $\gamma^{d+1}$ and negative ends  at $\gamma^d$ and $\gamma$.
\item $u_2$ is the union of the trivial cylinder $\R\times\gamma^d$ and an index two holomorphic plane with positive end at $\gamma$.
\end{itemize} 
Then a sequence of $J$-holomorphic cylinders $\{u(k)\}_{k=1,\ldots}$ in $\M^J(\gamma^{d+1},\gamma^d)/\R$ cannot converge in the sense of \cite{behwz} to $(u_1,u_2)$.
\end{proposition}

\subsection{Writhe bounds}

To prove Proposition~\ref{prop:nobadbreak}, we need to recall some results about the asymptotics of holomorphic curves.

Let $\gamma$ be an embedded Reeb orbit, and let $N$ be a tubular neighborhood of $\gamma$. We can identify $N$ with a disk bundle in the normal bundle to $\gamma$, and also with $\xi|_\gamma$.

Let $\zeta$ be a braid in $N$, i.e.\ a link in $N$ such that that the tubular neighborhood projection restricts to a submersion $\zeta\to\gamma$. Given a trivialization $\tau$ of $\xi|_\gamma$, one can then define the {\em writhe\/} $w_\tau(\zeta)\in\Z$. To define this one uses the trivialization $\tau$ to identify $N$ with $S^1\times D^2$, then projects $\zeta$ to an annulus and counts crossings of the projection with (nonstandard) signs. See \cite[\S2.6]{ir} or \cite[\S3.3]{bn} for details.

Now let $u$ be a $J$-holomorphic curve in $\R\times Y$. Suppose that $u$ has a positive end at $\gamma^d$ which is not part of a multiply covered component. Results of Siefring \cite[Cor.\ 2.5 and 2.6]{s1} show that if $s$ is sufficiently large, then the intersection of this end of $u$ with $\{s\}\times N\subset\{s\}\times Y$ is a braid $\zeta$, whose isotopy class is independent of $s$. We will need bounds on the writhe $w_\tau(\zeta)$, which are provided by the following lemma.

\begin{lemma}
\label{lem:positivewrithe}
Let $\gamma$ be an embedded Reeb orbit, let $u$ be a $J$-holomorphic curve in $\R\times Y$ with a positive end at $\gamma^d$ which is not part of a trivial cylinder or a multiply covered component, and let $\zeta$ denote the intersection of this end with $\{s\}\times Y$. If $s$ is sufficiently large, then the following hold:
\begin{description}
\item{(a)} $\zeta$ is the graph in $N$ of a nonvanishing section of $\xi|_{\gamma^d}$. Thus, using the trivialization $\tau$ to write this section as a map $\gamma^d\to\C\setminus\{0\}$, it has a well-defined winding number around $0$, which we denote by $\op{wind}_\tau(\zeta)$.
\item{(b)} $\op{wind}_\tau(\zeta) \le \floor{\op{CZ}_\tau(\gamma^d)/2}$.
\item{(c)} If $J$ is generic, $\op{CZ}_\tau(\gamma^d)$ is odd, and $\op{ind}(u)\le 2$, then equality holds in (b).
\item{(d)} $w_\tau(\zeta) \le (d-1)\op{wind}_\tau(\zeta)$.
\end{description}
\end{lemma}

\begin{proof}
Choose an identification $N\simeq (\R/\Z)\times D^2$ compatible with the trivialization $\tau$. The asymptotic behavior of holomorphic curves described in \cite{hwz1}, \cite{s1}, or \cite[Prop.\ 2.4]{ht2}, see the exposition in \cite[\S5.1]{bn}, implies the following. For $s_0>>0$, one can describe the intersection of this end of $u$ with $[s_0,\infty)\times N\subset [s_0,\infty)\times Y$ as the image of a map
\begin{align*}
[s_0,\infty)\times (\R/d\Z) &\longrightarrow \R\times (\R/\Z)\times D^2,\\
(s,t) &\longmapsto (s,\pi(t),\eta(s,t)),
\end{align*}
where $\pi:\R/d\Z\to\R/\Z$ denotes the projection, and $\eta$ is described as follows. Define the {\em asymptotic operator\/} $L$ from the space of smooth sections of $\xi|_{\gamma^d}$ to itself by
\[
L=J_{\pi(t)}\nabla_t,
\]
where $\nabla$ denotes the connection on $\xi|_{\gamma^d}$ defined by the linearized Reeb flow along $\gamma$. The operator $L$ is symmetric and so its eigenvalues are real. We now have
\begin{equation}
\label{eqn:eta}
\eta(s,t) = e^{-\mu s}\varphi(t) + O\left(e^{(-\mu-\varepsilon)s}\right),
\end{equation}
where $\mu>0$ is an eigenvalue of the asymptotic operator $L$, while $\varphi$ is a corresponding eigenfunction and $\varepsilon>0$.

It follows from the uniqueness of solutions to ODE's that the eigenfunction $\varphi$ is nowhere vanishing. Thus the eigenfunction $\varphi$ has a well-defined winding number around $0$, and together with \eqref{eqn:eta} this proves (a).

It is shown in \cite[\S3]{hwz2} that for each integer $n$, there are exactly two eigenvalues of $L$ for which eigenfunctions have winding number $n$. Here and below we count eigenvalues with multiplicity. Moreover, larger winding numbers correspond to smaller eigenvalues, and the largest possible winding number for a positive eigenvalue is $\floor{\op{CZ}_\tau(\gamma^d)/2}$. This implies (b).

To prove (c), note that the same argument in \cite[\S3]{hwz2} also shows that the smallest possible winding number of an eigenfunction of $L$ with negative eigenvalue is $\ceil{\op{CZ}_\tau(\gamma^d)/2}$. Since $\op{CZ}_\tau(\gamma^d)$ is assumed odd, we have a strict inequality $\floor{\op{CZ}_\tau(\gamma^d)/2} < \ceil{\op{CZ}_\tau(\gamma^d)/2}$. Consequently, the two (possibly equal) eigenvalues of $L$ whose eigenfunctions have winding number $\floor{\op{CZ}_\tau(\gamma^d)/2}$ are both positive. Thus, if equality does not hold in (b), then the eigenvalue $\mu$ in \eqref{eqn:eta} is not one of the two smallest positive eigenvalues of $L$ (counted with multiplicity as usual). Now, as pointed out by Chris Wendl, see \cite[Rmk.\ 3.3]{ht2}, one can use exponentially weighted Sobolev spaces to set up the moduli space of irreducible holomorphic curves in which the eigenvalue $\mu$ in \eqref{eqn:eta} is not one of the two smallest positive eigenvalues. If $J$ is generic, then somewhere injective holomorphic curves in this moduli space are cut out transversely, but the dimension of the moduli space is $2$ less than usual. Consequently there are no nontrivial somewhere injective holomorphic curves $u$ in this moduli space with $\op{ind}(u)\le 2$.

The analogue of (d) in an analytically simpler situation is proved in \cite[\S6]{pfh2}. 
This argument can be extended to the present case using the refined asymptotic analysis of Siefring \cite[Thms.\ 2.2 and 2.3]{s1}.
\end{proof}

\begin{remark}
\label{rem:improved}
Lemma~\ref{lem:positivewrithe}(b),(d) imply that
\[
w_\tau(\zeta) \le (d-1)\floor{\CZ_\tau(\gamma^d)/2}.
\]
In fact one can improve this to
\begin{equation}
\label{eqn:improved}
w_\tau(\zeta) \le (d-1)\floor{\CZ_\tau(\gamma^d)/2} - \op{gcd}\left(d,\floor{\CZ_\tau(\gamma^d)/2}\right)+1,
\end{equation}
see \cite{s2}. However we will not need this here.
\end{remark}

Symmetrically to Lemma~\ref{lem:positivewrithe}, we also have the following:

\begin{lemma}
\label{lem:negativewrithe}
Let $\gamma$ be an embedded Reeb orbit, let $u$ be a $J$-holomorphic curve in $\R\times Y$ with a negative end at $\gamma^d$ which is not part of a trivial cylinder or multiply covered component, and let $\zeta$ denote the intersection of this end with $\{s\}\times Y$. If $s<<0$, then the following hold:
\begin{description}
\item{(a)} $\zeta$ is the graph of a nonvanishing section of $\xi|_{\gamma^d}$, and thus has a well-defined winding number $\op{wind}_\tau(\zeta)$.
\item{(b)} $\op{wind}_\tau(\zeta) \ge \ceil{\op{CZ}_\tau(\gamma^d)/2}$.
\item{(c)} If $J$ is generic, $\op{CZ}_\tau(\gamma^d)$ is odd, and $\op{ind}(u)\le 2$, then equality holds in (b).
\item{(d)} $w_\tau(\zeta) \ge (d-1)\op{wind}_\tau(\zeta)$.
\end{description}
\end{lemma}

\subsection{Counting singularities}

We will also need the following inequality from intersection theory of holomorphic curves. As before, let $\gamma$ be an embedded Reeb orbit with tubular neighborhood $N$, and let $\tau$ be a trivialization of $\xi|_\gamma$.

\begin{lemma}
\label{lem:adjunction}
Let $u$ be a $J$-holomorphic curve in $[s_-,s_+]\times N$ with no multiply covered components and with boundary $\zeta_+-\zeta_-$ where $\zeta_\pm$ is a braid in $\{s_\pm\}\times N$. Then
\[
\chi(u) + w_\tau(\zeta_+) - w_\tau(\zeta_-) = 2\Delta(u) \ge 0,
\]
where $\chi(u)$ denotes the Euler characteristic of the domain of $u$, and $\Delta(u)$ is a count of the singularities of $u$ in $Y$ with positive integer weights.
\end{lemma}

\begin{proof}
This is proved similarly to the relative adjunction formula in \cite[Rmk.\ 3.2]{pfh2}. (The relative first Chern class and relative self-intersection pairing terms there are zero in our situation.)
\end{proof}

\subsection{Proof of Proposition~\ref{prop:nobadbreak}}

The proof of Proposition~\ref{prop:nobadbreak} has four steps.

{\em Step 1.\/} We begin with some index calculations.

First note that the hypothesis on $u_1$ forces $\gamma$ to be elliptic. The reason is that by \eqref{eqn:ind} and \eqref{eqn:CZiteration}, a pair of pants which is a branched cover of a hyperbolic trivial cylinder has index $1$.

Choose a trivialization $\tau$ of $\xi|_\gamma$. 
Let $\theta\in\R\setminus\Q$ denote the rotation angle of $\gamma$ with respect to $\tau$. Then by \eqref{eqn:CZiteration}, the Conley-Zehnder index of $\gamma$ is given by
\begin{equation}
\label{eqn:CZgamma}
\op{CZ}_\tau(\gamma) = 2\floor{\theta}+1.
\end{equation}
Likewise, the Conley-Zehnder indices of $\gamma^d$ and $\gamma^{d+1}$ are given by
\begin{align}
\label{eqn:CZgammad}
\op{CZ}_\tau(\gamma^d) &= 2\floor{d\theta}+1,\\
\label{eqn:CZgammad1}
\op{CZ}_\tau(\gamma^{d+1}) &= 2\floor{(d+1)\theta} + 1.
\end{align}
Our assumption that $\op{ind}(u_1)=0$ is now equivalent to
\begin{equation}
\label{eqn:indu10}
\floor{(d+1)\theta} = \floor{d\theta}+\floor{\theta}.
\end{equation}
Here we are using the index formula \eqref{eqn:ind} and the fact that $c_\tau(u_1)=0$.

{\em Step 2.\/}
We now assume that the proposition is false and set up some notation.

Recall that $u_2$ an equivalence class of holomorphic curves in $\R\times Y$, where two holomorphic curves are equivalent iff they differ by $\R$-translation in $\R\times Y$. Choose a representative of this equivalence class and still denote it by $u_2$. Translate the holomorphic curve $u_2$ downward if necessary so that Lemma~\ref{lem:positivewrithe} is applicable to $s\ge 0$. 

Let $N$ be a tubular neighborhood of the Reeb orbit $\gamma$.
Fix $\varepsilon>0$ so that $u_2(u_2^{-1} (\{0\}\times N))$ has distance at least $\varepsilon$ from $\gamma$.

Suppose to get a contradiction that there exists a sequence of $J$-holomorphic cylinders $\{u(k)\}$ in $\mc{M}^J(\gamma^{d+1},\gamma^d)/\R$ which converges in the sense of \cite{behwz} to the building $(u_1,u_2)$. Then for sufficiently large $k$, the equivalence class $u(k)$ in $\mc{M}^J(\gamma^{d+1},\gamma^d)/\R$ has a representative $u\in\mc{M}^J(\gamma^{d+1},\gamma^d)$ with the following properties:
\begin{description}
\item{(i)} $u^{-1}([0,\infty)\times Y)$ is an annulus with one puncture, which is mapped by $u$ to $[0,\infty)\times N$.
\item{(ii)} $u^{-1}((-\infty,0]\times Y)$ consists of a closed disk $D$ and a half-cylinder $C$.
\item{(iii)} $u(C)$ is contained in $(-\infty,0]\times N$, and $u(C)\cap(\{0\}\times N)$ is a braid $\zeta_1$ which projects to $\gamma$ with degree $d$ and has distance at most $\varepsilon/3$ from $\gamma$.
\item{(iv)} $u(D)\cap(\{0\}\times N)$ is a braid $\zeta_2$ which projects to $\gamma$ with degree $1$ and is within distance $\varepsilon/3$ of $u_2(u_2^{-1}(\{0\}\times N))$.
\end{description}
Also let $\zeta_+$ denote the braid corresponding to the positive end of $u$ at $\gamma^{d+1}$, and let $\zeta_-$ denote the braid corresponding to the negative end of $u$ at $\gamma^d$.

{\em Step 3.\/} We now obtain some inequalities from the previous lemmas.

Since $\zeta_1$ is within distance $\varepsilon/3$ of $\gamma$, while $\zeta_2$ has distance at least $2\varepsilon/3$ from $\gamma$, it follows that $\zeta_1\cup\zeta_2$ is a braid and
\[
w_\tau(\zeta_1\cup\zeta_2) = w_\tau(\zeta_1) + 2d\op{wind}_\tau(\zeta_2) + w_\tau(\zeta_2).
\]
Since the braid $\zeta_2$ projects to $\gamma$ with degree $1$, we have
\begin{equation}
\label{eqn:wtz20}
w_\tau(\zeta_2)=0.
\end{equation}
 By Lemma~\ref{lem:adjunction}, we have
\[
-1 + w_\tau(\zeta_+) - w_\tau(\zeta_1\cup\zeta_2) = 2\Delta_+ \ge 0
\]
where $\Delta_+$ denotes the count of singularities of $u$ in $[0,\infty)\times Y$.
By Lemma~\ref{lem:adjunction} again we have
\[
w_\tau(\zeta_1) - w_\tau(\zeta_-) = 2\Delta_- \ge 0
\]
where $\Delta_-$ denotes the count of singularities of $u|_C$. 
Putting the above four lines together gives
\begin{equation}
\label{eqn:combineadjunction}
-1 + w_\tau(\zeta_+) - 2d\op{wind}_\tau(\zeta_2) - w_\tau(\zeta_-) \ge 0.
\end{equation}

By Lemma~\ref{lem:positivewrithe}(b),(d) and equation \eqref{eqn:CZgammad1}, we have
\[
w_\tau(\zeta_+) \le d\floor{(d+1)\theta}.
\]
Since $J$ is assumed generic, applying Lemma~\ref{lem:positivewrithe}(b),(c) to $u_2$ and using equation \eqref{eqn:CZgamma} gives
\[
\op{wind}_\tau(\zeta_2) = \floor{\theta}.
\]
Finally, Lemma~\ref{lem:negativewrithe}(b),(d) and equation \eqref{eqn:CZgammad} give
\[
w_\tau(\zeta_-) \ge (d-1)(\floor{d\theta}+1).
\]
Putting the above three lines into \eqref{eqn:combineadjunction} gives
\begin{equation}
\label{eqn:combinelemmas}
d(\floor{(d+1)\theta}-2\floor{\theta}-1) - (d-1)\floor{d\theta} \ge 0.
\end{equation}

{\em Step 4.\/} We now complete the proof. Combining \eqref{eqn:indu10} with \eqref{eqn:combinelemmas} gives
\[
\floor{d\theta} \ge d(\floor{\theta} + 1).
\]
This is impossible because for any positive integer $d$ and real number $\theta$ we have $\floor{d\theta}\le d\theta < d(\floor{\theta}+1)$.  This contradiction completes the proof of Proposition~\ref{prop:nobadbreak}.

\begin{remark}
\label{remark:technical}
One should be able to show more generally that a sequence of cylinders cannot converge to a building as in case (iii) of Proposition~\ref{prop:1neg}(c) with $d_2$ arbitrary. To do so, one can follow the above argument, but one would need to generalize \eqref{eqn:wtz20} to get a formula for $w_\tau(\zeta_2)$. The required formula for $w_\tau(\zeta_2)$ would follow from:
\end{remark}

\begin{conjecture}
\label{conj:technical}
Under the assumptions of Lemma~\ref{lem:positivewrithe}(c), the inequality \eqref{eqn:improved} in Remark~\ref{rem:improved} is an equality.
\end{conjecture}

Conjecture~\ref{conj:technical} would follow if one could show that the first two coefficients in the asymptotic expansion of the end of the holomorphic curve are nonzero, instead of just the first coefficient as in Lemma~\ref{lem:positivewrithe}(c).

\section{Proof of the main theorem}
\label{sec:gluing}

We now prove Theorem~\ref{thm:main}, after some preliminaries on transversality and gluing.

\subsection{Automatic transversality}

We begin with an automatic transversality lemma. Much more general automatic transversality results are proved in \cite{wendl}, but we recall the proof of this simple lemma for the convenience of the reader.

Let $\lambda$ be a nondegenerate contact form on a three-manifold $Y$ and let $J$ be a $\lambda$-compatible almost complex structure. 
If $u:(\Sigma,j)\to(\R\times Y,J)$ is a $J$-holomorphic immersion (with ends at Reeb orbits as usual), with normal bundle $N$, then it has a deformation operator
\[
D_u: L^2_1(\Sigma,N)\to L^2(\Sigma,T^{0,1}\Sigma\tensor_\C N).
\]
The moduli space of holomorphic curves near $u$ is cut out transversely when $D_u$ is surjective, in which case the tangent space to the moduli space can be identified with $\Ker(D_u)$. Let $h_+(u)$ denote the number of ends of $u$ at positive hyperbolic orbits (including even covers of negative hyperbolic orbits).

\begin{lemma}
\label{lem:atrecall}
Let $\lambda$ be a nondegenerate contact form on a three-manifold $Y$ and let $J$ be a $\lambda$-compatible almost complex structure on $\R\times Y$. Let $u$ be a $J$-holomorphic immersion as above. If the domain $\Sigma$ is connected with genus $g(\Sigma)$, and if
\begin{equation}
\label{eqn:atcondition}
2g(\Sigma) - 2 + h_+(u) < \op{ind}(u),
\end{equation}
then $D_u$ is surjective (without any genericity assumption on $J$).
\end{lemma}

\begin{proof}
Suppose that $D_u$ is not surjective. Then there is a nonzero element $\psi$ of the kernel of the formal adjoint
\[
D_u^*: L^2_1(\Sigma,T^{0,1}\Sigma\tensor_\C N)\longrightarrow L^2(\Sigma,N).
\]
The Carleman similarity principle implies that the zeroes of $\psi$, if any, are isolated and have negative multiplicity. The asymptotic behavior of $\psi$ which we will describe in a moment implies that the zeroes of $\psi$ are contained in a compact set. It follows that the count of zeroes of $\psi$ with multiplicity, which we denote by $\#\psi^{-1}(0)$, is well defined and satisfies
\begin{equation}
\label{eqn:negativezeroes}
\#\psi^{-1}(0) \le 0.
\end{equation}

Let $\tau$ be a trivialization of $\xi$ over the Reeb orbits at which $u$ has ends; this induces a trivialization of $T^{0,1}\Sigma\tensor_\C N$ over the ends of $u$. By the definition of the relative first Chern class, we have
\[
\#\psi^{-1}(0) = c_1(T^{0,1}\Sigma\tensor_\C N,\tau) + \op{wind}_\tau(\psi)
\]
where $\op{wind}_\tau(\psi)$ denotes the sum over the positive ends of $u$ of the winding number of $\psi$ around the end with respect to the trivialization $\tau$, minus the corresponding sum over the negative ends of $u$. Since
\[
c_1(T^{0,1}\Sigma\tensor_\C N,\tau) = \chi(\Sigma) + c_1(N,\tau) = c_1(u^*T(\R\times Y),\tau) = c_1(u^*\xi,\tau),
\]
we can rewrite the previous formula as
\begin{equation}
\label{eqn:countzeroes}
\#\psi^{-1}(0) = c_\tau(\xi) + \op{wind}_\tau(\psi).
\end{equation}

Similarly to \eqref{eqn:eta}, on a positive end at a (possibly multiply covered) Reeb orbit $\gamma$, the section $\psi$ has the asymptotic behavior
\[
\psi(s,t) = e^{\mu s}\varphi(t) + O\left(e^{(\mu-\varepsilon)s}\right),
\]
where $\mu$ is now a {\em negative\/} eigenvalue of the asymptotic operator $L$ associated to $\gamma$. It follows from Lemma~\ref{lem:negativewrithe}(b) that the winding number of $\psi$ around this end is at least $\ceil{\CZ_\tau(\gamma)/2}$. Together with an analogous calculation for the negative ends, it follows that if $u$ has positive ends at $\alpha_1,\ldots,\alpha_k$ and negative ends at $\beta_1,\ldots,\beta_l$, then
\[
\op{wind}_\tau(\psi) \ge \sum_{i=1}^k \ceil{\CZ_\tau(\alpha_i)/2} - \sum_{j=1}^l \floor{\CZ_\tau(\beta_j)/2}.
\]
Since the Conley-Zehnder index of a Reeb orbit is even exactly when that Reeb orbit is positive hyperbolic, we deduce that
\begin{equation}
\label{eqn:windbound}
2\op{wind}_\tau(\psi) \ge k + l - h_+(u) + \sum_{i=1}^k\CZ_\tau(\alpha_i) - \sum_{j=1}^l\CZ_\tau(\beta_j).
\end{equation}
Combining \eqref{eqn:negativezeroes}, \eqref{eqn:countzeroes}, and \eqref{eqn:windbound} with the index formula \eqref{eqn:ind} gives
\[
\op{ind}(u) + 2 - 2g(\Sigma) - h_+(u) \le 0.
\]
This is the negation of the hypothesis \eqref{eqn:atcondition}.
\end{proof}

\subsection{Transversality}

We can now establish the transversality needed to define cylindrical contact homology. Some of the following lemma was also proved in \cite[\S2.3]{momin}.

\begin{lemma}
\label{lem:at}
Let $Y$ be a closed three-manifold with a nondegenerate contact form $\lambda$. Let $J$ be a generic $\lambda$-compatible almost complex structure on $\R\times Y$. Then:
\begin{description}
\item{(a)}
For any Reeb orbits $\gamma_+$ and $\gamma_-$, the moduli space $\mc{M}^J_1(\gamma_+,\gamma_-)/\R$ is a $0$-manifold cut out transversely.
\item{(b)}
If $\gamma_+$ and $\gamma_-$ are good Reeb orbits, then the moduli space $\mc{M}^J_2(\gamma_+,\gamma_-)/\R$ is a $1$-manifold cut out transversely.
\item{(c)} If $\gamma_+$ and $\gamma_-$ are good, then the function
\[
d:\mc{M}^J_2(\gamma_+,\gamma_-)/\R \longrightarrow \Z^{>0},
\]
which associates to each cylinder its covering multiplicity, is locally constant.
\end{description}
\end{lemma}

\begin{proof}
We know from \cite{dragnev} that if $J$ is generic, then all irreducible somewhere injective $J$-holomorphic curves are cut out transversely.  It is also shown in \cite[Thm.\ 4.1]{ht2} that if $J$ is generic, then all irreducible somewhere injective $J$-holomorphic curves of index $\le 2$ are immersed. To prove the lemma, we will show that if $J$ satisies the above two properties, then (a), (b) and (c) hold.

Let $u\in\mc{M}_k^J(\gamma_+,\gamma_-)$ be a $J$-holomorphic cylinder with $\op{ind}(u)=k\in\{1,2\}$, and assume that $\gamma_+$ and $\gamma_-$ are good when $k=2$. Then $u$ is a $d$-fold cover of a somewhere injective cylinder $\overline{u}\in\mc{M}^J(\overline{\gamma_+}, \overline{\gamma_-})$. Since $\op{ind}(u)\le 2$, it follows from Lemma~\ref{lem:cylindercoverindex}(a) that $\op{ind}(\overline{u})\le 2$. Then $\overline{u}$ is immersed, and consequently $u$ is also immersed.

Since $u$ is a cylinder of positive index, the inequality \eqref{eqn:atcondition} must hold, so Lemma~\ref{lem:atrecall} implies that $D_u$ is surjective.
This does not yet prove (a) and (b), because a priori $\mc{M}^J_k(\gamma_+,\gamma_-)$ might only be an orbifold near $u$. To prove that $\mc{M}^J_k(\gamma_+,\gamma_-)$ is in fact a manifold near $u$, we need to further show that the order $d$ group of deck transformations of $u$ over $\overline{u}$ acts trivially on $\op{Ker}(D_u)$. For this purpose, and also to prove (c), it will suffice to show that every element of $\op{Ker}(D_u)$ is pulled back from an element of $\op{Ker}(D_{\overline{u}})$. To prove this last claim, it is enough to show that $\op{ind}(u)=\op{ind}(\overline{u})$.

We know that $u$ and $\overline{u}$ both have index $1$ or $2$. We just need to rule out the case where $\op{ind}(u)=2$ and $\op{ind}(\overline{u})=1$. In this case, one of $\overline{\gamma_+}$, $\overline{\gamma_-}$ is positive hyperbolic, while the other is elliptic or negative hyperbolic. The elliptic case is impossible because then all covers of $\overline{u}$ have odd index. Thus one of $\overline{\gamma_+}$, $\overline{\gamma_-}$ is positive hyperbolic and the other is negative hyperbolic. Then $\op{ind}(u)=d$, so $d=2$. This contradicts the assumption that $\gamma_+$ and $\gamma_-$ are good.
\end{proof}

\subsection{Gluing}

\begin{lemma}
\label{lem:gluing}
Assume $J$ is generic. Suppose that $u_+\in\mc{M}_1^J(\gamma_+,\gamma_0)/\R$ and $u_-\in\mc{M}_1^J(\gamma_0,\gamma_-)/\R$. Assume that $\gamma_+$ and $\gamma_-$ are good Reeb orbits, and let $k=\op{gcd}(d(u_+),d(u_-))$. Then:
\begin{description}
\item{(a)}
There are exactly $kd(\gamma_0)/d(u_+)d(u_-)$ ends of the moduli space $\mc{M}_2^J(\gamma_+,\gamma_-)/\R$ that converge to the building $(u_+,u_-)$.
\item{(b)}
Each such end consists of cylinders with $d=k$.
\end{description}
\end{lemma}

\begin{proof}
(a)
Fix a point $p$ on the image of $\gamma_0$ in $Y$. Choose representatives
\[
\phi_\pm:\R\times S^1\to\R\times Y
\]
of $u_\pm$ such that
\begin{equation}
\label{eqn:p}
\lim_{s\to \mp \infty}\pi_Y(\phi_\pm(s,0)) = p.
\end{equation}
To glue, one first translates $\phi_+$ up and $\phi_-$ down, then ``preglues'' them by patching them together using cutoff functions, and finally uses the contraction mapping theorem to perturb the preglued cylinder to a holomorphic cylinder. See e.g.\ \cite[\S5]{ht2}, which carries out a general gluing construction of which the above is a special case\footnote{In \cite{ht2} it is assumed that the holomorphic curves to be glued are not multiply covered, but the construction works the same way for multiply covered cylinders.}.

Every gluing is obtained by taking some pair $(\phi_+,\phi_-)$ satisfying \eqref{eqn:p} and applying the above construction, cf.\ \cite[\S7]{ht2}. Up to $\R$-translation of the domain and target, which does not affect the end of the index two moduli space attained by gluing, there are $d(\gamma_0)/d(u_\pm)$ distinct parametrizations $\phi_\pm$ of $u_\pm$ satisfying \eqref{eqn:p}. This gives $d(\gamma_0)^2/d(u_+)d(u_-)$ pairs $(\phi_+,\phi_-)$ satisfying \eqref{eqn:p}. Observe that the cyclic group $\Z/d(\gamma_0)$ acts on the set of pairs $(\phi_+,\phi_-)$ satisfying \eqref{eqn:p} by rotating the $S^1$ coordinate of $\phi_+$ and $\phi_-$. Two pairs $(\phi_+,\phi_-)$ glue to the same end of the index two moduli space if and only if they are in the same orbit of this $\Z/d(\gamma_0)$ action. Now $j\in\Z/d(\gamma_0)$ fixes the pair $(\phi_+,\phi_-)$ if and only if $j$ is a multiple of $d(\gamma_0)/d(u_+)$ and $d(\gamma_0)/d(u_-)$, i.e.\ if and only if $j$ is a multiple of $d(\gamma_0)/k$. Thus each orbit of the $\Z/d(\gamma_0)$ action has cardinality $d(\gamma_0)/k$, so the number of orbits is
\[
\frac{d(\gamma_0)^2/d(u_+)d(u_-)}{d(\gamma_0)/k} = \frac{kd(\gamma_0)}{d(u_+)d(u_-)}.
\]

(b) Since $\gamma_+$ and $\gamma_-$ are good, we know from Lemma~\ref{lem:at}(c) that on each end of the moduli space of index $2$ cylinders converging to $(u_+,u_-)$, the covering multiplicity $d$ is constant. Now $d(u_+)$ and $d(u_-)$ must both be multiples of $d$. Therefore $d$ is a divisor of $k$. To complete the proof of (b), it is enough to show that $d$ is a multiple of $k$. To do so, note that each cylinder $u_\pm$ is a $k$-fold cover of a cylinder $\widehat{u_\pm}$. It is then enough to show that every end of the moduli space of index $2$ cylinders converging to $(u_+,u_-)$ is obtained by taking $k$-fold covers of an end of the moduli space of index $2$ cylinders converging to $(\widehat{u_+},\widehat{u_-})$.

By Lemma~\ref{lem:cylindercoverindex}(a) we have $\op{ind}(\widehat{u_\pm})=1$, so we can apply part (a) to the pair $(\widehat{u_+},\widehat{u_-})$. This tells us that the number of ends of the moduli space of index $2$ cylinders converging to $(\widehat{u_+},\widehat{u_-})$ is
\[
\frac{d(\gamma_0)/k}{d(\widehat{u_+})d(\widehat{u_-})} = \frac{kd(\gamma_0)}{d(u_+)d(u_-)}.
\]
By part (a) again, this agrees with the number of ends of the moduli space of index $2$ cylinders converging to $(u_+,u_-)$. Thus all of the latter ends are accounted for by $k$-fold covers of ends converging to $(\widehat{u_+},\widehat{u_-})$.
\end{proof}

\subsection{Proof of Theorem~\ref{thm:main}}

Assume that $\lambda$ satisfies the hypotheses of the theorem and that $J$ is generic.

By Lemma~\ref{lem:at}(a), each moduli space $\mc{M}_1^J(\gamma_+,\gamma_-)/\R$ of index $1$ cylinders is a $0$-manifold, which can be oriented by a choice of coherent orientations as in \cite{bm}. By Proposition~\ref{prop:1neg}(b), $\mc{M}_1^J(\gamma_+,\gamma_-)/\R$ is compact, hence finite. Thus the operator $\delta$ in \eqref{eqn:delta} is defined.

To prove that $\delta\kappa\delta=0$, suppose that $\gamma_+$ and $\gamma_-$ are good Reeb orbits. We know from Lemma~\ref{lem:at}(b) that the moduli space $\mc{M}_2^J(\gamma_+,\gamma_-)/\R$ of index $2$ cylinders is an oriented $1$-manifold, and the covering multiplicity $d$ is constant on each component. We claim that $\mc{M}_2^J(\gamma_+,\gamma_-)/\R$ has a compactification to a compact oriented $1$-manifold $\overline{\mc{M}_2^J(\gamma_+,\gamma_-)/\R}$, obtained by attaching one boundary point to each end, such that
\begin{equation}
\label{eqn:sumX}
\sum_{X\in \pi_0\left( \overline{\mc{M}_2^J(\gamma_+,\gamma_-)/\R}\right) } \frac{\#\partial X}{d(X)} = \langle\delta\kappa\delta\gamma_+,\gamma_-\rangle.
\end{equation}
Here $\#\partial X$ denotes the signed count of boundary points of the component $X$, which of course is zero. Thus equation \eqref{eqn:sumX} implies that $\delta\kappa\delta=0$.

To prove \eqref{eqn:sumX}, note that by Proposition~\ref{prop:1neg}(c), each end of $\mc{M}_2^J(\gamma_+,\gamma_-)/\R$ limits to a building as in case (ii) or (iii) of Proposition~\ref{prop:1neg}(c). But in fact case (iii) cannot happen, because then the Reeb orbit $\gamma^{d_2}$ is contractible and has $\CZ(\gamma^{d_2})=3$, so the hypothesis (*) of the theorem implies that $d_2=1$, and Proposition~\ref{prop:nobadbreak} gives a contradiction.
Consequently, each end of $\mc{M}_2^J(\gamma_+,\gamma_-)/\R$ limits to a building $(u_+,u_-)$, where $u_+\in\mc{M}_1^J(\gamma_+,\gamma_0)/\R$ and $u_-\in\mc{M}_1^J(\gamma_0,\gamma_-)/\R$ for some Reeb orbit $\gamma_0$. Given such a pair $(u_+,u_-)$, let $G(u_+,u_-)$ denote its contribution to the left hand side of \eqref{eqn:sumX}. We need to show that
\begin{equation}
\label{eqn:G}
G(u_+,u_-) = \left\{\begin{array}{cl} \frac{\epsilon(u_+)\epsilon(u_-)d(\gamma_0)}{d(u_+)d(u_-)}, & \mbox{if $\gamma_0$ is good},\\
0, & \mbox{if $\gamma_0$ is bad}.
\end{array}\right.
\end{equation}

If $\gamma_0$ is bad, let $\overline{u_\pm}$ denote the somewhere injective curve underlying $u_\pm$. By Lemma~\ref{lem:cylindercoverindex}(b), the negative end of $\overline{u_+}$ and the positive end of $\overline{u_-}$ are both at bad Reeb orbits. In particular, $d(\gamma_0)/d(u_\pm)$ is even. Thus the least common multiple of $d(u_+)$ and $d(u_-)$ divides $d(\gamma_0)/2$. It then follows from Lemma~\ref{lem:gluing}(a) that there are an even number of ends of the moduli space $\mc{M}_2^J(\gamma_+,\gamma_-)$ that converge to the building $(u_+,u_-)$. These ends are related to each other by the $\Z/d(\gamma_0)$ action described in the proof of Lemma~\ref{lem:gluing}(a). By \cite[Thm.\ 3]{bm}, shifting by $1\in\Z/d(\gamma_0)$ switches the sign of the corresponding end, i.e.\ the sign of the corresponding boundary point of the index two moduli space. Thus half of the ends have positive sign and half have negative sign, so $G(u_+,u_-)=0$.

If $\gamma_0$ is good, then by Lemma~\ref{lem:gluing}, the number of ends of the moduli space of index $2$ cylinders converging to $(u_+,u_-)$, divided by their multiplicity, is $d(\gamma_0)/d(u_+)d(u_-)$. By \cite{bm}, each end has sign $\epsilon(u_+)\epsilon(u_-)$. This implies \eqref{eqn:G} and completes the proof of Theorem~\ref{thm:main}.

\end{document}